\documentclass[]{amsart}
\usepackage{amscd,amsthm,amssymb,amsfonts,amsmath,euscript}
\usepackage{multirow}
\theoremstyle{plain}
\newtheorem{thm}{Theorem}[section]
\newtheorem{lemma}[thm]{Lemma}
\newtheorem{prop}[thm]{Proposition}
\newtheorem{cor}[thm]{Corollary}

\theoremstyle{definition}
\newtheorem{defn}[thm]{Definition}

\theoremstyle{remark}
\newtheorem{remark}[thm]{Remark}

\newcommand{\nc}{\newcommand}

\newtheorem*{thm*}{{\bf Theorem}}
\newtheorem*{lemma*}{{\bf Lemma}}
\newtheorem*{prop*}{{\bf Proposition}}
\newtheorem*{cor*}{{\bf Corollary}}
\newtheorem*{defn*}{{\bf Definition}}

\def\makeop#1{\expandafter\def\csname#1\endcsname
  {\mathop{\rm #1}\nolimits}\ignorespaces}
\makeop{Hom}   \makeop{End}   \makeop{Aut}   \makeop{Isom}  \makeop{Pic} 
\makeop{Gal}   \makeop{ord}   \makeop{Char}  \makeop{Div}   \makeop{Lie} 
\makeop{PGL}   \makeop{Corr}  \makeop{PSL}   \makeop{sgn}   \makeop{Spf}
\makeop{Spec}  \makeop{Tr}    \makeop{Nr}    \makeop{Fr}    \makeop{disc}
\makeop{Proj}  \makeop{supp}  \makeop{ker}   \makeop{im}    \makeop{dom}
\makeop{coker} \makeop{Stab}  \makeop{SO}    \makeop{SL}    \makeop{SL}
\makeop{Cl}    \makeop{cond}  \makeop{Br}    \makeop{inv}   \makeop{rank}
\makeop{id}    \makeop{Fil}   \makeop{Frac}  \makeop{GL}    \makeop{SU}
\makeop{Nrd}   \makeop{Sp}    \makeop{Tr}    \makeop{Trd}   \makeop{diag}
\makeop{Res}   \makeop{ind}   \makeop{depth} \makeop{Tr}    \makeop{st}
\makeop{Ad}    \makeop{Int}   \makeop{tr}    \makeop{Sym}   \makeop{can}
\makeop{length}\makeop{SO}    \makeop{torsion} \makeop{GSp} \makeop{Ker}
\makeop{Adm}   \makeop{Mat}
\def\makebb#1{\expandafter\def
  \csname bb#1\endcsname{{\mathbb{#1}}}\ignorespaces}
\def\makebf#1{\expandafter\def\csname bf#1\endcsname{{\bf
      #1}}\ignorespaces} 
\def\makegr#1{\expandafter\def
  \csname gr#1\endcsname{{\mathfrak{#1}}}\ignorespaces}
\def\makescr#1{\expandafter\def
  \csname scr#1\endcsname{{\EuScript{#1}}}\ignorespaces}
\def\makecal#1{\expandafter\def\csname cal#1\endcsname{{\mathcal
      #1}}\ignorespaces} 
\def\doLetters#1{#1A #1B #1C #1D #1E #1F #1G #1H #1I #1J #1K #1L #1M
                 #1N #1O #1P #1Q #1R #1S #1T #1U #1V #1W #1X #1Y #1Z}
\def\doletters#1{#1a #1b #1c #1d #1e #1f #1g #1h #1i #1j #1k #1l #1m
                 #1n #1o #1p #1q #1r #1s #1t #1u #1v #1w #1x #1y #1z}
\doLetters\makebb   \doLetters\makecal  \doLetters\makebf
\doLetters\makescr 
\doletters\makebf   \doLetters\makegr   \doletters\makegr

     \def\qed{\qedmark\medbreak}%
\def\qedmark{{\enspace\vrule height 6pt width 5pt depth 1.5pt}}%

\normalsize

\makeop{Bl}

\def\Spec{{\rm Spec}\,}

\def\Fpbar{\overline{\bbF}_p}
\def\Fp{{\bbF}_p}
\def\Fq{{\bbF}_q}

\def\Qp{{\bbQ}_p}

\def\Zp{{\bbZ}_p}

\newcommand{\Z}{\mathbb Z}
\newcommand{\Q}{\mathbb Q}

\newcommand{\C}{\mathbb C}
\newcommand{\F}{\mathbb F}

\newcommand{\npr}{\noindent }

\newcommand{\<}{\langle}   %\< is not defined yet.
\renewcommand{\>}{\rangle} %\> is already defined.

\newcommand{\isoto}{\stackrel{\sim}{\longrightarrow}}
\nc{\embed}{\hookrightarrow}
\newcommand{\ch}{characteristic }
\newcommand{\ac}{algebraically closed }
\newcommand{\dieu}{Dieudonn\'{e} }

\nc{\ol}{\overline}
\nc{\wt}{\widetilde}
\nc{\opp}{\mathrm{opp}}
\def\ul{\underline}

\makeop{Ram}
\makeop{Rep}

\begin{document}
\renewcommand{\thefootnote}{\fnsymbol{footnote}}
\setcounter{footnote}{-1}
\numberwithin{equation}{section}
%\numberwithin{section}{chapter}

%\usepackage[notref,notcite]{showkeys}

\title{A note on supersingular abelian varieties} 
\author{Chia-Fu Yu}
\address{
Institute of Mathematics, Academia Sinica and NCTS  \\
6th Floor, Astronomy Mathematics Building \\
No. 1, Roosevelt Rd. Sec. 4 \\ 
Taipei, Taiwan, 10617} 
\email{chiafu@math.sinica.edu.tw}

% \address{
% The Max-Planck-Institut f\"ur Mathematik \\
% Vivatsgasse 7, Bonn \\
% Germany 53111} 
% \email{chiafu@mpim-bonn.mpg.de}

%\date{June 23, 2000}

\date{\today}

\begin{abstract}
In this note we show that any supersingular abelian variety is 
isogenous to a superspecial abelian variety without increasing 
field extensions. The proof uses minimal isogenies and 
the Galois descent. We then construct a superspecial abelian variety 
which not directly defined over a finite field. This answers
negatively to a question of the author [J. Pure Appl. Alg., 2013] 
concerning of endomorphism algebras occurring in Shimura curves.  
Endomorphism algebras of supersingular elliptic curves over 
an arbitrary field are also investigated. We correct a main
result of the author's paper [Math. Res. Let., 2010].  
\end{abstract} 

\maketitle

%\tableofcontents   % Table of Contents

\section{Introduction}
\label{sec:01}

Throughout this note, $p$ denotes a prime number and all ground fields
considered are of \ch $p$, unless specifically stated otherwise. 
We discuss  
endomorphism algebras and fields of definition of supersingular
 isogeny classes.
% of  of supersingular abelian varieties. 
% in characteristic $p$.
% All fields considered in this note are of \ch $p$, where $p$ is
% a prime.
%Let $\bar k$ denote an algebraic closure of a field $k$, or an \ac
%field. 
% Let $k$ denote a field of \ch $p$ and $\bar k$ an algebraic closure of
% $k$, or an \ac field of \ch $p$. 
There are several equivalent definitions for supersingular abelian
varieties. % (cf. \cite[Section.~1.6]{li-oort}. 
One definition states that an abelian variety $A$ in 
\ch $p$ is {\bf supersingular} if its associated $p$-divisible group
$A[p^\infty]$ has only one slope $1/2$. 
This definition 
is well-defined in the sense that it does not depend on the ground field 
over which $A$ is defined. However, the definition involves 
the \dieu theory and the Manin-\dieu classification of $p$-divisible
groups up to isogeny over an \ac field $k$.    
Oort \cite{oort:subvar} showed that any supersingular abelian isogeny
over $k$ is 
isogenous to a product of supersingular elliptic curves $E_1\times
\dots \times E_g$ over $k$, which provides an alternative but much simpler
definition. 
Our discussion will rely on the following result due to 
Deligne, Ogus and Shioda (cf.~\cite[Section 1.6, p.~13]{li-oort}). 

\begin{thm}\label{Ogus_et}
For any integer $g\ge 2$ and any
supersingular elliptic curves $E_i$ 
over an \ac field $k$ for $1\le i\le 2g$, 
one has $E_1\times \dots \times E_g
\simeq E_{g+1}\times \dots \times E_{2g}$. 
\end{thm}

One can choose $E_i=E_0$ for a supersingular
elliptic curve $E_0$ which is defined over a finite field, or even over
$\Fp$, with Frobenius endomorphism $\pi$ satisfying $\pi^2+p=0$ if $g>1$. 
Thus, any supersingular abelian variety $A$ of dimension $g>1$ 
over $k$ can be obtained by an isogeny  
\begin{equation}
  \label{eq:1.1}
  \varphi: E_0^g \otimes_{\Fp} k \to A. 
\end{equation}
 
Over a not necessarily \ac field $k$, 
the parameterization as (\ref{eq:1.1}) 
is more involved. By (\ref{eq:1.1}), 
an abelian variety $A$ over $k$ is 
supersingular if and only if there exists an finite field extension
$k_1/k$ and a $k_1$-isogeny 
\begin{equation}
  \label{eq:1.2}
  \varphi: E_0^g \otimes k_1 \to A \otimes_k k_1.
\end{equation}
% $\varphi: E_0^g \otimes k_1 \to A \otimes_k
% k_1$   over $k_1$.

In order to obtain all supersingular abelian varieties over a given
field $k$, the alterations from $E_0^g \otimes k$ by extending 
a finite field extension and an isogeny are all necessary. 
For example, if $A$ is $k$-simple, then $A$
cannot be obtained in the way (\ref{eq:1.2}) without base change over
a field extension. Using the Honda-Tate theorem \cite{tate:ht}, 
it is easy to construct $\Fq$-simple supersingular
abelian varieties over $\Fq$. On the other hand, the ``moduli space'' of
supersingular abelian varieties of dimension $g>1$ 
has positive dimension, while the
locus which consists of superspecial abelian varieties 
is always zero-dimensional. This indicates that a modification by an
isogeny is also necessary.
Recall that an abelian variety $A$
over $k$ is said to be {\bf superspecial} if $A\otimes_k \bar k$ is
isomorphic to a product of supersingular elliptic curves, where $\bar
k$ denotes an algebraic closure of $k$.  

It is natural to ask whether any
supersingular abelian variety over $k$ is $k$-isogenous to a superspecial
abelian variety over $k$, and whether any supersingular abelian variety
is isogenous to $A'$ which can be {\it directly defined} over a finite
field. By ``directly defined'' we mean that there exists 
an abelian variety $A_0$ over a
finite field $\Fq$ contained in $k$ and a $k$-isomorphism $A_0\otimes_{\Fq}
k \simeq A'$. This is different from a weaker property that 
the field of moduli of $A'$ is a finite field. 

If the second question had an affirmative answer, then
one can determine all endomorphism algebras of supersingular
abelian varieties over any field $k$. Indeed, any endomorphism algebra
of this kind can be also realized by a supersingular abelian variety  
which is defined over a finite field. 
% the endomorphism algebra of $A$ also occurs as the endomorphism
% algebra of a supersingular abelian variety over a finite field
% $\Fq$. 
By a theorem of Manin and Oort, 
any Weil $q$-number $\pi$ which corresponds to a supersingular isogeny
class is of the form $\sqrt{q} \zeta$, where $\zeta$
is a root of unity  % (a theorem of Manin and Oort, 
; see \cite[p.~116]{oort:subvar} (also \cite[Theorem 2.8]{yu:QM}).   
Using Tate's theorem on endomorphism algebras \cite[p.~142]{tate:eav}, 
one shows that except the
case where $\pi=\sqrt{q}\not\in \Q$,  
the endomorphism algebra of any simple supersingular abelian variety 
over $\Fq$ either an abelian CM field or a quaternion division 
algebra over a cyclotomic field \cite[Sect. 3]{xue-yang-yu:sp_as}. 

In this note We prove the following result. 

% However, the second question has the
% negative answer (see a proof in Section ?). One main result of this
% note shows that the first question has the affirmative answer. 

% Let $k$ denote a field of \ch $p$ and $\bar k$ an algebraic closure of
% $k$. 
% Recall that an elliptic curve over $k$ 
% is supersingular if it  has no non-zero $\bar k$-valued
% $p$-torsion points. 
% % where $\bar k$ denotes an algebraic closure of $k$.  
% An abelian variety over $k$ 
% is {\bf supersingular} if it is isogenous over a product of
% supersingular elliptic curves over $\bar k$;
% it is said to be
% {\bf superspecial} if it is isomorphic to a product of supersingular
% elliptic curves over $\bar k$. 

% Particularly any $g$-dimensional superspecial abelian
% variety over $\bar k$ is defined over $\F_{p}$ if $g>1$ 
% (and over $\F_{p^2}$ if $g=1$ due to Deuring). 
% For convenience of discussion, 
% we say an abelian variety $A$ over $k$ {\bf trained} 
%   if there is an
%   abelian variety $A_0$ over a finite field $k_0\subset k$ 
%   and a $k$-isomorphism $A_0\otimes k\simeq A$. 
%   Do not confuse this
%   with the weaker property that the field of moduli of $A$ is a finite
%   field. 

% Consider the following two statements:

% \begin{itemize}
% \item [(A)] Any supersingular abelian variety over $k$ is isogenous to a
% superspecial abelian variety over $k$ (without increasing a field 
% extension).
% \item [(B)] Any superspecial abelian variety is trained. 
% \end{itemize}

% A main result of this note shows the following.

\begin{thm}\label{1.1}
  (1) Let $A$ be a supersingular abelian variety over a field $k$. 
   There exists a superspecial abelian variety $C$ over $k$ and a
  $k$-isogeny $A\to C$.

  (2) There exists a supersingular abelian variety $A$ over some field
  $k$ such that for any $k$-isogeny $A\sim A'$ of abelian varieties, 
  $A'$ is not directly defined over a finite field.  
\end{thm}

By Theorem~\ref{1.1} (1), any supersingular abelian variety over 
$k$ can be
 parametrized by a pair $(A_0,\varphi)$, where $A_0$ is a superspecial
abelian variety over $k$ and $\varphi:A_0\to A$ is a
$k$-isogeny. 
By Theorem~\ref{Ogus_et}, when $g>1$, any such $A_0$ is a
$k$-form of $E_0^g\otimes k$ and these $k$-forms 
are classified by the Galois cohomology
$H^1(k, \GL_g(\calO))$, where $\calO:=\End(E_0\otimes k_s)$ and $k_s$
is a separable closure of $k$. The 
endomorphism algebra $E^0(E\otimes k_s)$ is a definite
quaternion $\Q$-algebra ramified at $\{p,\infty\}$ 
and $\calO$ is a maximal order. 
In the case that $k\supset \F_{p^2}$, the Galois group
 $\Gamma_k:=\Gal(k_s/k)$ acts trivially on $\GL_g(\calO)$ and we have 
\[ H^1(k,
 \GL_g(\calO))=\Hom(\Gamma_k,\GL_g(\calO))/\GL_g(\calO), \]
 which
 parameterizes conjugacy classes of finite subgroups $H$ in
 $\GL_g(\calO)$ which occur as quotients of $\Gamma_k$. The inverse
 Galois problem enters to play a role.    

It follows from Theorem~\ref{1.1} that there exists a superspecial
abelian variety which is not directly defined over a finite field; see
Remark~\ref{3.2} (2). 
We construct such an example in Section~\ref{sec:03}. 
This answers negatively to a
question in \cite[(Q), p.~912]{yu:QM}. The construction uses
arithmetic properties of definite quaternion $\Q$-algebras 
and an explicit computation of Galois cohomology.

% It is well known that 
% not every supersingular abelian variety over $k$ or even over
% $\bar k$ is trained. Theorems~\ref{Ogus_et} and \ref{1.1} suggest 
% that 
% uniqueness of superspecial abelian varieties of dimension $>1$, 
% Statement (B) is false in general. Otherwise every supersingular
% abelian variety would be isogenous to a product of supersingular elliptic
% curves over the same ground field. We confirm this by 
% constructing counterexamples to Statement (B) (Section~\ref{sec:03}).

% A basic method
% are Galois cohomology we also use results on the inverse Galois problem .

% Combining this and the result on Shioda et al, at
% least one of the following statement must fail.

% \
% It follows that Statement (B) is false in general. 
% Counterexamples to Statement (B) are constructed in
% Section~\ref{sec:03}.  

A related topic is a theorem of Grothendieck, which 
states that if $A$ is an
abelian variety with smCM over a field $k$, 
then there is a
finite field extension $k_1/k$, an abelian variety $A_0$ over a finite
field $k_0\subset k_1$ and a $k_1$-isogeny $\varphi:
A\otimes_k k_1\to A_0\otimes_{k_0} k_1$; 
see \cite[pp.~220/221]{mumford:av}, \cite[Theorem 1.1]{oort:cm} and
\cite[Theorem 1.4]{yu:cm}. Recall that an abelian variety $A$ is said
to have {\bf smCM} if 
the endomorphism algebra $\End^0(A)$ of $A$
contains a commutative semi-simple $\Q$-subalgebra 
of degree $2\dim A$. In this case $L$ can be chosen to be a CM
algebra. Similar to the case of supersingular abelian varieties, 
the conditions ``up to isogeny'' and 
``up to a finite extension'' in Grothendieck's theorem are all
necessary. For example, a geometric generic supersingular abelian surface 
is not defined over a finite field and hence 
the first condition is necessary. 
See Section~\ref{sec:03} that 
the second condition ``up to a finite extension'' is also necessary.

% Galois descent and Theorem~\ref{Ogus_et} play important roles in the
% proof of Statement (A). 

As a consequence of Theorem~\ref{1.1}, 
one can in principle compute the endomorphism algebra of 
any supersingular abelian variety. Exploring the case where $g=1$, we
obtain the following result. There is a new example in the case where
$p\not \equiv 1 \pmod {12}$. 
% For higher dimensional cases, see a forthcoming paper . 

% The theory also allows us to compute the endomorphism algebras of
% abelian varieties twisted by $1$-cocycles, not just classifying them. 
% This leads to the following result about
% endomorphism algebras of supersingular elliptic curves, over an
% arbitrary field of \ch $p$. It differs sightly from well-known results
% where $k$ is a finite field 
% that $\Q$ also occurs in the endomorphism algebras in question.
%  of supersingular elliptic curves.  
%The case where $k$ is either a finite
%field or an \ac field is certainly well known. 

% Let $B_{p,\infty}$ 
% denote the (unique up
% to isomorphism) quaternion 
% $\Q$-algebra ramified exactly at $\{p,\infty\}$. 

\begin{thm}\label{thm:end-ss}\ 
  \begin{enumerate}
  \item [(1)] If $p\not \equiv 1 \pmod {12}$, then 
    there is a supersingular elliptic
    curve $E$ over a field $k$ so that $\End^0(E)=\Q$.
  \item [(2)] If $p\equiv 1 \pmod {12}$, then for 
    any supersingular elliptic curve
    $E$ over an arbitrary field 
    $k$, one has $\End^0(E)\neq \Q$. In other words,
    $\End^0(E)$ is isomorphic to a semi-simple
    $\Q$-subalgebra of degree $2$ or $4$ in 
    the definite quaternion $\Q$-algebra $B_{p,\infty}$ ramified
    exactly at $\{p,\infty\}$.   
  \end{enumerate}
\end{thm}

The proof of Theorem~\ref{thm:end-ss} (1) is given by an explicit
construction, which depends on a construction of certain Galois 
field extensions. The involved inverse Galois problem 
(IGP) is fortunately rather easy to solve.  

Our proof of Theorem~\ref{thm:end-ss} (2) also gives 
all possible $\Q$-algebras that can occur as the endomorphism algebras of
supersingular elliptic curves over an arbitrary field when $p\equiv 1
\pmod {12}$. Furthermore, all of them also occur  
as the endomorphism algebras of those which are  defined
over finite fields. 

% However, all of them, no matter these curves are trained or not, 
% occur in the endomorphism algebras of those
% over finite fields.  

% though some supersingular elliptic curves in
% Theorem~\ref{thm:end-ss} (2) may not be trained, the endomorphism
% algebras of 
% all of them are the same as the endomorphism algebras of those defined
% over finite fields. 

\section{Minimal isogenies over perfect fields}
\label{sec:M}

In this section, we show the existence of minimal isogenies for
abelian varieties and $p$-divisible groups over perfect
fields. This generalizes previous results of the author \cite{yu:endo}
where the ground field is assumed to be algebraically closed.  
A special case of this results will be used in the proof of
Theorem~\ref{1.1} (1).

Let $k$ be a perfect field of \ch $p$. 
% unless specifically stated otherwise.
Let $W:=W(k)$ be the ring of Witt
vectors over $k$, and $B(k)$ the fraction field of $W(k)$. Let
$\sigma$ be the Frobenius map on $W$ and $B(k)$, respectively. 
%For any $W$-module $M$ and any subset $S\subset M$, 
%denote by $\<S\>_W$
%the $W$-submodule generated by $S$. Similarly, $\<S\>_{B(k)}\subset
%M\otimes \Q_p$ denotes the vector subspace over $B(k)$ generated by
%$S$. 
We use the covariant \dieu theory \cite{zink:cartier}.
\dieu modules considered here are assumed to be finite and free as 
$W$-modules.

To each rational number $0\le \lambda\le 1$, one associates a pair
$(a,b)$ of coprime
non-negative integers so that $\lambda=b/(a+b)$.    
For each pair $(a,b)\neq (0,0)$ as above, write
$M_{(a,b)}$ for the \dieu module $W[F,V]/(F^a-V^b)$ over $k$ of slope 
$\lambda=b/(a+b)$.

We write a Newton polygon of slopes in $[0,1]$ as a finite
formal sum $\beta=\sum_{i=1}^s r_i (a_i,b_i)$,  
or express it in terms of a slope sequence $\beta=(\lambda_1^{(r_1)},
\dots, \lambda_s^{(r_s)})$,
where each $0\le \lambda_i\le 1$ is a rational number, $r_i\in \bbN$ is a
positive integer, and $(a_i,b_i)$ is the pair associated to
$\lambda_i$ (By convention, the multiplicity of the slope 
$\lambda_i$ is $(a_i+b_i) r_i$). 
Let $M$ be a \dieu module over $k$ and put
$N:=M\otimes_W B(k)$.  
By the Manin-\dieu Theorem (\cite{manin:thesis}, Chap. II,
``Classification Theorem'', p.~35), the isocrystal
$N':=N\otimes_{B(k)} B(k')$, where $k'\supset k$ is an \ac overfield,  
admits a unique decomposition
\begin{equation}
  \label{eq:M.1}
  N'=\oplus_i N'_{\lambda_i}, \quad N'_{\lambda_i} \simeq 
  (M_{(a_i,b_i)}\otimes_W B( k'))^{r_i}, 
\end{equation}
for some Newton polygon $\beta=\sum_i r_i(a_i,b_i)$. The invariant
$\beta$ is independent of the choice of $k'$ and is called the 
{\bf Newton polygon} of
$M$. By the existence of slope filtration for \dieu modules 
and the splitting theorem up to isogeny  
\cite[Theorem 2.5.1]{katz:slope}, the decomposition (\ref{eq:M.1}) is
actually defined over $B(k)$. Namely, one has a unique decomposition
of $N$ into isotypic components   
\begin{equation}
  \label{eq:M.2}
  N=\oplus_i N_{\lambda_i}, \quad N_{\lambda_i}\otimes_{B(k)}
  B(k')=N'_{\lambda_i}. 
\end{equation}
See Zink~\cite{zink:slope} and Oort and Zink \cite{oort-zink} for a
far generalization of the slope filtration for $p$-divisible groups 
over normal base schemes. 

% This is the Manin-\dieu Theorem (\cite{manin:thesis}, Chap. II,
% ``Classification Theorem'', p.~35) when the ground field $k$ is
% algebracially closed. 
% For any perfect field $k$, this 
% follows from the existence of the slope filtration for \dieu
% modules; see Zink~\cite{zink}, Oort and Zink \cite{oort-zink} for the
% generalization for families of $p$-divisible groups.   

Let $(a,b)\neq (0,0)\in (\Z_{\ge 0})^2$ be a pair as before, and put 
$n=a+b$. Denote by
$\bfM_{(a,b)}$ the \dieu module over $\Fp$ generated by
the elements $e_i$, for $i\in \Z_{\ge 0}$, with relation $e_{i+n}=pe_i$, 
as a $\Zp$-module, 
and with operations $Fe_i=e_{i+b}$ and $Ve_i=e_{i+a}$ for all $i\in
\Z_{\ge 0}$. Clearly, $F^n=p^b$ and $V^n=p^a$ on $\bfM_{(a,b)}$. 
For any Newton polygon $\beta=\sum_{i=1}^s r_i (a_i, b_i)$, define 
\[ \bfM(\beta):=\bigoplus_{i=1}^s
\bfM_{(a_i,b_i)}^{r_i}. \]

Let $M_\lambda$ be an isoclinic \dieu module over $k$ of slope 
slope $\lambda=b/(a+b)$. There are two integers $x$ and $y$ 
such that $bx+ay=1$. We define an $\sigma^{x-y}$-linear operator on
$N_\lambda:=M_\lambda\otimes_W B(k)$ by
\begin{equation}
  \label{eq:M.3}
  \Phi_\lambda=F^x V^y: N_\lambda\to N_\lambda. 
\end{equation}
If $M_\lambda=\bfM_{(a,b)}$, then one has $\Phi_\lambda(e_i)=e_{i+1}$
for all $i$. 

\begin{defn}
  A \dieu module $M$ over $k$ with Newton polygon $\beta$
  is said to be {\bf minimal} if for any
  \ac field $k'\supset k$, one has $M\otimes_W W(k')\simeq
  \bfM(\beta)\otimes_{\Zp} W(k')$. 
\end{defn}
  
\begin{lemma}\label{M.2}
  Let $M$ be a \dieu module over an \ac field $k$ with Newton polygon
  $\beta$. The following are equivalent:

  \begin{itemize}
  \item[(a)] There is an isomorphism $M\simeq 
  \bfM(\beta)\otimes_{\Zp} W(k)$ of \dieu modules. 
\item[(b)] The endomorphism ring $\End(M)$ is a maximal order of the
  endomorphism algebra $\End^0(M)=\End(M)\otimes_{\Zp} \Qp$ of $M$.
 
\item [(c)] $M$ is the direct sum of its isotypic components
  $M_\lambda$ and each component $M_\lambda$ is a minimal \dieu module. 

  \end{itemize}
If $M$ is isoclinic with slope $\lambda=b/(a+b)$, then the statements
(a), (b), or (c) is equivalent to 
  \begin{itemize}

\item [(d)] $F^{a+b}M=p^b M$ and $\Phi_\lambda(M)\subset M$, where
  $\Phi_\lambda$ is defined in (\ref{eq:M.3}).
\end{itemize}
\end{lemma}
\begin{proof}
  See \cite[Lemmas 3.3 and 3.4]{yu:endo}. \qed
\end{proof}

By Lemma~\ref{M.2}, a \dieu module $M$ over $k$ is
minimal if and only if so is $M\otimes_W W(k')\simeq
\bfM(\beta)\otimes_{\Zp} W(k')$ for one \ac field $k'\supset k$.

\begin{defn}[cf.~{\cite[Section 1]{li-oort}} and 
  {\cite[Section 4]{yu:endo}}]\label{M.3}
  Let $X$ be a $p$-divisible group over a field $k_1$ of \ch $p$. 

\npr (1) We say that $X$ is {\bf minimal} if the \dieu module of
$X\otimes_{k_1} \bar 
k_1$ satisfies one of the equivalent conditions in Lemma~\ref{M.2}.

\npr (2) The {\bf minimal isogeny} of
  $X$ is a pair $(X_0, \varphi)$, where $X_0$ is a minimal $p$-divisible
  group over $k_!$, and $\varphi:X_0\to X$ is an isogeny over $k_1$ such
  that for any other pair $(X_0',\varphi')$ as above, there exists an
  isogeny $\rho: X'_0\to X_0$ such that
  $\varphi'=\varphi\circ\rho$. Note that the morphism $\rho$ is unique
  if it exists.
\end{defn}

\begin{lemma}\label{M.4}
  Let $M$ be a \dieu module over $k$ with Newton
  polygon $\beta$. Then there exists a unique smallest minimal \dieu
  module $M^{\rm min}$ over $k$ containing $M$. 
  Dually, there is a unique largest
  minimal \dieu submodule $M_{\rm min}$ of $M$.
\end{lemma}
\begin{proof}
  We first show that if $M_1$ and $M_2$ are minimal \dieu  modules (of full
  rank) in $N:=M\otimes B(k)$ then so are $M_1+M_2$ and $M_1\cap
  M_2$. 
  % are also minimal. 
  For any \ac overfield $k'$, put $M_i':=M_i\otimes
  W(k')$ for $i=1,2$. Clearly we have 
  $(M_1+M_2)\otimes W(k')=M_1'+M_2'$ and
  $(M_1\cap M_2)\otimes W(k')=M_1'\cap M_2'$. Thus, we are reduced to
  the case where $k$ is algebraically closed and this follows from
  Lemma~\ref{M.2} as both $M_1+M_2$ and $M_1\cap M_2$ satisfy the
  criterion (d). Therefore, the uniqueness of $M^{\rm min}$ 
  has been proved and it remains to show that there is a minimal 
  \dieu module over $k$ containing $M$. 

  Put $M_\lambda=M\cap N_\lambda$ and by (\ref{eq:M.2}) we have $M\subset
  \oplus_{\lambda} M_\lambda$, where $N_\lambda$ is the isotypic
  component of $N$ of slope $\lambda$. It then suffices to show that
  each $M_\lambda$ is contained in a minimal \dieu module, and we can
  assume that $M=M_\lambda$ is isoclinic of slope $\lambda=b/(a+b)$. 
  Let $P(M)$ be the
  $W$-submodule generated by $M$ which is stable under the operators 
  $F^{a+b}p^{-b}$ and $\Phi_\lambda$. By \cite[Lemma
  4.2]{yu:endo}, the \dieu module  $M\otimes_W W(k')$ is
  contained in a minimal \dieu module, which
  is stable under these operators. It follows that 
  $P(M)$ is a $W$-module of finite rank and it is a \dieu module. 
  As $P(M)\otimes
  W(k')$ is a \dieu module stable under $F^{a+b}p^{-b}$ and
  $\Phi_\lambda$, it follows from Lemma~\ref{M.2} that $P(M)$ is
  minimal. \qed       
\end{proof}

The minimal \dieu module $M_{\rm min}$ (resp.~$M^{\rm min}$)
constructed in Lemma~\ref{M.4} 
is called the {\bf minimal \dieu
submodule} (resp.~{\bf overmodule}) of $M$.
% the module $M^{\rm min}$ is called the {\it minimal
% \dieu overmodule} of $M$. 
By Lemma~\ref{M.4}, we have  

\begin{cor}\label{M.5}
  The minimal isogeny of any $p$-divisible group $X$ over $k$ exists. 
\end{cor}

Let $\calO$ be an order of a finite-dimensional semi-simple algebra
over $\Q_p$. A $p$-divisible $\calO$-module is a pair $(X,\iota)$,
where $X$ is a $p$-divisible group and $\iota:\calO\to \End(X)$ is a
ring monomorphism. 

\begin{prop}\label{M.6}
  Let $(X,\iota)$ be a $p$-divisible $\calO$-module over 
  $k$ and let
  $\varphi:X_0\to X$ be the minimal isogeny of $X$ over $k$. Then
  there is a unique ring monomorphism $\iota_0:\calO\to \End(X_0)$
  such that $\varphi$ is $\calO$-linear. 
\end{prop}
\begin{proof}
  The same statement is proved in \cite[Prop.~4.8]{yu:endo} where the
  ground field is assumed to be algebraically closed. The same proof
  also works for the present situation. \qed
\end{proof}

\begin{remark}\label{M.7}
  In a recent preprint \cite{harashita:satu_NP}, S.~Harashita shows that any
  $p$-divisible group over an arbitrary field $k$ of \ch $p$ 
  is $k$-isogenous to a
  minimal $p$-divisible group. His result can improve
  Corollary~\ref{M.5} and Proposition~\ref{M.6} without the perfectness
  assumption of the ground field.  
\end{remark}

\section{Proof of Theorem~\ref{1.1} (1)}
\label{sec:02}

% We will prove the following result. 

%\begin{prop}\label{21}
%   Let $A$ be a supersingular abelian variety over a field $k$ of \ch
%   $p$. There exists a superspecial abelian variety $C$ over $k$ and a
%   $k$-isogeny $A\to C$. 
% \end{prop}

We recall briefly the Galois descent. 
Let $X_0$ be a quasi-projective
algebraic variety over a field $k$ and $K/k$ a finite Galois
extension with group $\Gal(K/k)$. Write $X_{0K}:=X_0\otimes_k K$.
Let $E(K/k, X_0)$ denote the set of $k$-isomorphism classes of
$K/k$-forms $X$ of $X_0$, i.e.~$X\otimes_k K$ is isomorphic to
$X_{0K}$ over $K$. There is a natural bijection 
\begin{equation}
  \label{eq:2.1}
  E(K/k, X_0) \isoto H^1(\Gal(K/k), \Aut(X_{0K})), \quad [X_0]\mapsto
  [{\rm id}_{X_0}]. 
\end{equation}
Suppose $X_0'$ is a quasi-projective variety over $k$
and $\eta: X_{0K} \isoto X_0'\otimes_k K$ is a
$K$-isomorphism. For any $\sigma\in \Gal(K/k)$, 
put $a_\sigma:=\eta^{-1}\circ \sigma(\eta)\in \Aut(X_{0K})$. 
Then $\{a_\sigma\}$ is a $1$-cocycle with values in
$\Aut(X_{0K})$, i.e.~one has $a_{\sigma\tau}=a_\sigma \sigma(a_\tau)$, for
$\sigma,\tau\in \Gal(K/k)$. The class of $\{a_\sigma\}$ in 
$H^1(\Gal(K/k), \Aut(X_{0K}))$  
is uniquely determined by the $k$-isomorphism class of
$X_0'$. Conversely, the Galois descent asserts that 
any $1$-cocycle $\{a_\sigma\}$ with values in
$\Aut(X_{0K})$ is represented by a pair $(X_0', \eta)$ as above.
% such that  
%$a_\sigma:=\eta^{-1}\circ \sigma(\eta)$ for all $\sigma\in
%\Gal(K/k)$. 
This describes the bijection (\ref{eq:2.1}).

% Namely, there is a natural bijection 
%\begin{equation}
%  \label{eq:2.1}
%   E(K/k, X_0) \isoto H^1(\Gal(K/k), \Aut(X_{0K})), \quad [X_0]\mapsto
%   [{\rm id}_{X_0}], 
%\end{equation}
%where .

Let $\End(X)$ denote the monoid of all
$k$-morphisms from a variety $X/k$ to itself.

\begin{lemma}\label{end}
  Let $X_0'\in E(K/k, X_0)$ and $\{\xi_\sigma\}$ a $1$-cocycle
  associated to $X_0'$. Then there is an isomorphism of monoids 
\begin{equation}
  \label{eq:2.2}
  \begin{split}
  \End(X_0')\simeq \{a\in \End(X_{0}\otimes K)\mid \xi_\sigma
  \sigma(a) 
  \xi_\sigma^{-1} =a, \forall\,
  \sigma\in \Gal(K/k) \}.    
  \end{split}
\end{equation}
\end{lemma}
\begin{proof}
  Choose an isomorphism $\eta: X_{0K}\simeq X'_{0K}$ such that 
  $\xi_\sigma=\eta^{-1} \sigma(\eta)$.
  % $\{\xi_\sigma\}$. 
  We have an isomorphism
  $\alpha:\End(X'_{0K})\simeq \End(X_{0K})$ by $\alpha(a'):=\eta^{-1} a'
  \eta$, i.e.~there is a commutative diagram 
\[ 
\begin{CD}
  X_{0K} @>\eta>> X'_{0K} \\
  @VV{\alpha(a')}V @VV{a'}V \\
 X_{0K} @>\eta>> X'_{0K}. \\
\end{CD}
\]
Put $a=\alpha(a')=\eta^{-1} a' \eta$ and
$b=\alpha(\sigma(a'))=\eta^{-1} \sigma(a') \eta$. One has
\[ \sigma(a)=\sigma(\eta^{-1}) \sigma(a') \sigma(\eta)=
\sigma(\eta^{-1}) \eta [\eta^{-1} \sigma(a') \eta] \eta^{-1}
\sigma(\eta)=\xi_\sigma^{-1}\cdot b \cdot
\xi_\sigma. \]  
That is, if one identifies $\End(X'_{0K})$ with $\End(X_{0K})$ by
$\alpha$, then the new $\Gal(K/k)$-action on $\End(X_{0K})$ induced
from the form $X'_0$ is given by $\sigma: a\mapsto \xi_\sigma \sigma(a)
\xi_\sigma^{-1}$. Thus,  
\begin{equation}
  \label{eq:2.3}
  \begin{split}
  \End(X_0')&=\{a'\in \End(X'_0\otimes K)\mid \sigma(a')=a', \forall\,
   \sigma \in \Gal(K/k)\}\\
   & \isoto \{a\in \End(X_0\otimes K)\mid \xi_\sigma \sigma(a)
   \xi_\sigma^{-1} 
   =a, \forall\,
   \sigma\in \Gal(K/k) \}.    
  \end{split} 
\end{equation}
\end{proof}

\begin{lemma}\label{22}
  Let $A$ be a supersingular abelian variety of dimension $g$ over a
  field $k$ of \ch $p$. There is a finite purely
  inseparable extension field $L/k$ and an $L$-isogeny 
  $A_L=A\otimes_k L \to B$,
  where $B$ is a superspecial abelian variety over $L$.
\end{lemma}
\begin{proof}
%  This is a known fact; a proof is provided for the reader's
%  convenience. 
  Let $k'$ be the perfect closure of $k$. It suffices to show that
  there is a $k'$-isogeny $\varphi: A_{k'}\to B$ for a superspecial
  abelian variety $B$ over $k'$ because $\ker \varphi$ is defined over a
  finite extension $L$ of $k$ in $k'$ and hence both $B$ and $\varphi$
  are defined over $L$, which is purely inseparable over $k$. 
  Let $M$ be the covariant \dieu module
  of $A_{k'}$. By Lemma~\ref{M.4}, $M$ is contained in a superspecial \dieu
  module $M'$ over $k'$. Therefore there is an (necessarily
  superspecial) abelian variety $B$
  over $k'$ and a $k$-isogeny $\varphi: A_{k'}\to B$ which realizes
  the chain of \dieu modules $M\subset M'$. \qed    
\end{proof}

\begin{lemma}\label{23}
  Let $A_1$ and $A_2$ be two abelian varieties over $k$ and $L/k$ a
  primary field extension (i.e.~$k$ is separably algebraically closed
  in $L$). Then we have an
  isomorphism
\[ \Hom_k(A_1,A_2) \isoto \Hom_L(A_{1,L}, A_{2,L}). \] 
\end{lemma}
\begin{proof}
  See \cite[Lemma 1.2.1.2]{chai-conrad-oort}. A key ingredient is
  that the Hom-scheme $\ul {\Hom}_k(A_1,A_2)\to \Spec k$ is
  unramified. This follows from the rigidity of endomorphisms of
  abelian schemes. \qed
\end{proof}

\begin{lemma}\label{24}
  Let $L/k$ be a
  finite purely inseparable extension field and $B$ a superspecial
  abelian variety over $L$ of dimension $g\ge 2$. Then there exists
  an abelian variety $B'$ over $k$ and an $L$-isomorphism $B'_L\simeq B$.
\end{lemma}
\begin{proof}
  Take any superspecial abelian variety $A$ over $k$ of dimension
  $g$. For example let $A=E^g\otimes_{\Fp}k$, where $E$ is a
  supersingular elliptic curve over $\Fp$. By Theorem~\ref{Ogus_et}, 
  there is a finite field extension $K$ over $L$ and 
  a $K$-isomorphism $\varphi: 
  B_K\simeq A_L\otimes_L K$. By Lemma~\ref{23},  the isomorphism $\varphi$
is defined over the maximal separable extension $L_s$ of $L$ in
$K$. Replacing $K$ by $L_s$ and $L_s$ by its Galois closure 
we may assume that $K$ is finite Galois over $L$. We review $B$ as a
$K/L$-form of $A_L$ and there is a corresponding 1-cocycle 
$\{\xi_\sigma\}$ of $\Gal(K/L)$ with
values in $\Aut(A_K)$. Let $K_1$ be the maximal separable field
extension of $k$ in $K$; it is the field generated by
sufficiently high $p$-th powers of elements of $K$ over $k$. Then
$K/K_1$ is a purely inseparable field extension of degree $[L:k]$, $L$
and $K_1$ are linearly disjoint over $k$, and the restriction gives an
isomorphism $\Gal(K/L)\simeq \Gal(K_1/k)$:
\[ 
\begin{CD}
  L @>{\rm Gal}>> K \\
  @AA{\rm insep}A @AA{\rm insep}A \\
  k @>>{\rm Gal}> K_1\, . 
\end{CD} \]
 Identifying $\Gal(K/L)$
with $\Gal(K_1/k)$, and $\Aut(A_K)$ with $\Aut(A_{K_1})$ due to
Lemma~\ref{23}, we regard $\{\xi_\sigma\}$ as a 1-cocycle of
$\Gal(K_1/k)$ with values in $\Aut(A_{K_1})$. By the Galois descent
there is an abelian variety $B'$ over $k$ corresponding to 
$\{\xi_\sigma\}$. As $B'_L$ and $B$ give rise to the same $1$-cocycle,
they are isomorphic. \qed
%there is an $L$-isomorphism $B'_L\simeq B$. 
%Clearly $B'$ is a superspecial abelian variety.        
\end{proof}

{\bf Proof of Theorem~\ref{1.1} (1)} 
There is nothing to prove if $g=\dim(A)=1$; we may assume that $g\ge
2$. By Lemma~\ref{22}, there is a superspecial abelian variety $B$
over a finite purely inseparable field extension $L/k$ and an
$L$-isogeny $A_L\to B$. By Lemma~\ref{24}, there is a superspecial
abelian variety $C$ over $k$ and an $L$-isomorphism $B\simeq
C_L$. Thus, there is an $L$-isogeny $\varphi:A_L \to
C_L$. By Lemma~\ref{23}, $\varphi$ is defined over
$k$. \qed  

\section{Construction of examples for Theorem~\ref{1.1} (2)}
\label{sec:03}

%\subsection{}
%\label{sec:31}

We shall construct a supersingular abelian variety satisfying the
property in Theorem~\ref{1.1} (2). More precisely,
for each prime $p$, we find a supersingular elliptic curve in \ch $p$ 
which is not directly defined over a finite field.  

We need some fine arithmetic results of definite quaternion
$\Q$-algebras; for example see \cite[Proposition 3.1, p.~145]{vigneras}.  
Let $\grA$ be a definite quaternion $\Q$-algebra, $\calO$ a
maximal order in $\grA$ and $G=\calO^\times$. Except for 
$\grA=B_{2, \infty}$ or $\grA=B_{3,\infty}$, the group $G$ is cyclic
of order $2, 4$ or $6$.
% where $B_{\infty,p}$ denotes the quaternion algebra over $\Q$
%  ramified exactly at $\{\infty,p\}$. 

When $\grA=B_{2,\infty}=(-1,-1/\Q)$, the class number of $\calO$ is 
one and any maximal order is conjugate to $\calO$. 
Using the mass formula, $\# G=24$. Furthermore, one has
\[ G=E_{24}=\left \{\pm 1, \pm i, \pm j, \pm k, \frac{ \pm 1 \pm i \pm
    j \pm k}{2} \right \}. \]

When $\grA=B_{3,\infty}=(-1,-3/\Q)$, the class number of $\calO$
is one and any maximal order is conjugate to $\calO$. Using the mass
formula, $\# G=12$. The group $G$ is isomorphic to the binary 
dihedral group of order 12
\[ G=T_{12}=\<a,b\,|\, a^6=1, b^2=a^3, bab^{-1}=a^{-1} \>. \]
 
Choose a supersingular elliptic curve $E_0$ over
$\F_{p^2}$ with Frobenius endomorphism $\pi_{E_0}=-p$. Then 
$\End^0(E_0)\simeq B_{p,\infty}$ and 
$\End(E_0)$ is a maximal order. 
%$\calO$ in $\End^0(E_0)$. being
%identified with $B_{\infty,p}$. 
Put $G=\Aut(E_0)$.

Choose an integer $m>1$ with $m\mid p^2-1$ and an element $\zeta\in
\F_{p^2}^\times$ of order $m$. Such an integer $m$ always exists for
any prime $p$; for example let $m=3$ if $p=2$, and $m=2$ if $p$ is
odd. Consider $k:=\F_{p^2}(T)$ and
$K:=\F_{p^2}(T^{1/m})$, where $T$ is a variable. 
Then $K/k$ is a cyclic extension with Galois
group $\Gal(K/k)=\<\sigma_m\>$, where $\sigma_m(T^{1/m})=\zeta 
T^{1/m}$. Since all endomorphisms of $E_0\otimes \Fpbar$ are defined
over $\F_{p^2}$,  the group $\Gal(K/k)$ acts trivially 
on $G=\Aut(E_0\otimes K)$. 
% and $G=\Aut(E_0\otimes K)^\times$..
 
By (\ref{eq:2.1}), the set $E(K/k,E_0\otimes k)$ is
in bijection with 
\begin{equation}
  \label{eq:3.1}
  H^1(\Gal(K/k), G)\simeq \Hom(\Gal(K/k), G)/G\simeq
\Hom(\Z/m\Z,G)/G,
\end{equation}
where $G$ acts $\Hom(\Z/m\Z,G)$ by conjugation. Note that the set
$E(K/k, E_0\otimes k)$ contains a non-trivial class.

\begin{prop}\label{3.1} \

{\rm (1)} Let $E/k$ be an elliptic curve in a non-trivial
class in $E(K/k, E_0\otimes k)$. Then $E$ is not $k$-isogenous to an
elliptic curve $E'/k$ which is directly defined over
$\F_{p^2}$. In particular, $E$ is not directly defined over $\F_{p^2}$.

{\rm (2)} If $m=3$ and $p=2$, then $\End^0(E)\simeq \Q(\zeta_3)$, where
$\zeta_n$ denotes a primitive $n$th root of unity. 

{\rm (3)} If $m=2$ and $p$ is odd, then $\End^0(E)\simeq B_{p,\infty}$. 
\end{prop}
\begin{proof}
(1) Suppose contrarily that there
is an elliptic curve $E'_0$ over $\F_{p^2}$ and a
$k$-isogeny $\varphi: E'_0\otimes k\to E$. 
Choose a $K$-isomorphism $\alpha_K:E_0\otimes_{\F_{p^2}} K\isoto
E\otimes_k K$. We get a $K$-isogeny 
$\beta_K: E_0'\otimes_{\F_{p^2}} K\to E_0\otimes_{\F_{p^2}} K$ such
that $\varphi_K=\alpha_K\circ \beta_K$, where
$\varphi_K=\varphi\otimes K$. 
Clearly $K/\F_{p^2}$ is primary, by Lemma~\ref{23}, $\beta_K$ is
defined over $\F_{p^2}$. As $\beta_K$ and $\varphi_K$ are defined over
$k$, the isomorphism $\alpha_K$ is defined over $k$ but $E$ is not
$k$-isomorphic to $E_0\otimes_{\F_{p^2}} k$, a contradiction.  

% Since $E\otimes_k K\isoto
% E_0\otimes_{\F_{p^2}} K$, 
% there is a $K$-isomorphism
% $\varphi: E_0'\otimes_{\F_{p^2}} K\isoto E_0\otimes_{\F_{p^2}} K$. 
% Clearly $K/\F_{p^2}$ is primary, by Lemma~\ref{23} $\varphi$ is
% defined over $\F_{p^2}$. This gives an isomorphism 
% $E\simeq E_0'\otimes k\simeq
% E_0\otimes_{\F_{p^2}} k$, a contradiction.

(2) Let $\{\xi\}$ be a 1-cocycle representing $E$, which can be viewed
    as an element in $\Hom(\Z/m\Z,G)$. Then $\xi_{\sigma_m}=\omega$,
    where $\omega\in G$ is an element of order $3$. By Lemma~\ref{end}
    we have 
\[ \End^0(E)\simeq \{a\in \End^0(E_0\otimes K)\mid \xi_\sigma
\sigma(a) \xi_\sigma^{-1}=a, \ \forall\, \sigma\in \Gal(K/k)\}. \]
Therefore, $\End^0(E)$ is isomorphic to the centralizer of
$\Q(\omega)$ in $B_{2,\infty}$ and $\End^0(E)=\Q(\omega)$.

(3) Using the same argument as (2), $\End^0(E)$ is isomorphic to the 
centralizer of $\Q$ in $B_{p,\infty}$ and hence $\End^0(E)\simeq
    B_{p,\infty}$. \qed      
\end{proof}

Theorem~\ref{1.1} (2) follows from Proposition~\ref{3.1} (1).

\begin{remark}\label{3.2}
(1)  It follows from Proposition~\ref{3.1} that for any prime $p$, there
  exists an elliptic curve over $k$ with smCM that is not
  $k$-isogenous to an elliptic curve which is directly defined over a
  finite field. This shows that 
  the condition ``up to a finite extension'' in
  Grothendieck's theorem is necessary. 

(2) It follows from Theorem~\ref{1.1} that there exists a superspecial
    abelian variety which is not directly defined over a finite
    field. Indeed, by Theorem~\ref{1.1} (2), there is a supersingular
    abelian variety $A$ over some field $k$ which is not $k$-isogenous
    to an abelian variety $A'$ directly defined over a finite
    field. Let $A_0$ be a superspecial abelian variety over $k$ which
    is $k$-isogenous to $A$, then $A_0$ can not be directly defined
    over a finite field by the property of $A$. Proposition~\ref{3.1}
    exhibits an example of superspecial abelian variety which is not
    directly defined over a finite field. 
\end{remark}

\section{Proof of Theorem~\ref{thm:end-ss}}
\label{sec:04}

\subsection{Part (1): case $p=2, 3$}
\label{sec:4.1}

We shall need some results arising from the inverse Galois problem. 
A useful result is that any finitely generated
infinite field $L$ over its prime field % $\bbP$ 
is {\bf Hilbertian}
(cf. \cite[p.~298]{schinzel:poly}),
that is, the Hilbert irreducibility theorem for $L$ holds. In
particular the rational function field $\F_{q}(T)$ is Hilbertian.

Let $E_0/\F_{p^2}$ and $k=\F_{p^2}(T)$ be as in Section~\ref{sec:03}. 
Let $p=2$ and  $Q=\{\pm 1, \pm i, \pm j, \pm k\}  \subset
G:=\Aut(E_0)=E_{24}$ the
quaternion subgroup of order $8$. We know that there is a generic
Galois extension $L/k(s)$ with Galois group $Q$ (see
\cite[Theorem 6.1.12, p.~140]{jensen-ladet-yui}), where $s$ is a 
variable. By the Hilbert
irreducibility theorem there is a finite Galois extension $K/k$
with Galois group $Q$. 
Choose an isomorphism $\xi:\Gal(K/k)\isoto Q\subset G$
and let $E\in E(K/k, E_0\otimes k)$ be the member corresponding to 
the $1$-cocycle $\{\xi_\sigma\}$ 
(noting that $\Gal(K/k)$ acts trivially on $G$).
% Let $\xi:\Gal(K/k)\isoto Q\subset G$
% be an isomorphism, which defines a $1$-cocycle $\{\xi_\sigma\}$ as
% $\Gal(K/k)$ acts trivially on $G$. 
% Let $E\in E(K/k, E_0\otimes k)$ be the member
% corresponding to $\{\xi_\sigma\}$. 
By the same computation
as in Proposition~\ref{3.1}, 
$\End^0(E)$ is isomorphic to the
centralizer of $Q$ in $B_{2,\infty}$. Clearly $\Q(Q)=B_{2,\infty}$ and
$\End^0(E)=\Q$.

% is a 
% semi-simple $\Q$-subalgebra strictly containing $\Q(i)$, one has 
% $\Q(Q)=B_{2,\infty}$ and $\End^0(E)=\Q$.

Now $p=3$. Similarly using the Hilbert irreducibility theorem there is 
a finite Galois extension $K/k$ with Galois
group $S_3$; see \cite[Remark, p.~29]{jensen-ladet-yui}.
%there is a generic Galois extension $L/k(s)$ with Galois
%group $D_{12}$ , there is 
%a finite Galois extension $K/k$ with Galois
%group $D_{12}$. 
Choose an isomorphism $\xi:\Gal(K/k)\isoto S_3\subset T_{12}=G$
and let $E\in E(K/k, E_0\otimes k)$ be the member corresponding to 
the $1$-cocycle $\{\xi_\sigma\}$. Using the same argument,
$\End^0(E)$
is isomorphic to the centralizer of $S_{3}$ in $B_{3,\infty}$ and 
hence $\End^0(E)=\Q$.

\subsection{Part (1): case $p>3$.} 
\label{sec:4.2}

Since $p\not\equiv 1\pmod {12}$, one has
$p \equiv 3 \pmod 4$ or $p\equiv 2 \pmod 3$. 
Put $m=4$ if $p\equiv 3 \pmod 4$ and $m=6$ if $p\equiv 2 \pmod
3$ (choose any $m\in\{4,6\}$ when $p\equiv 11 \pmod {12}$). 
There is a supersingular elliptic curve $E_0$ over $\Fp$ with
$\Aut(E_0\otimes \F_{p^2})=C_m=\<\eta\>$, where $C_m$ is the cyclic
group of order $m$ and $\eta$ is a generator. 
For example, let $E_0$ be the elliptic curve defined 
by $y^2=x^3-x$ or $y^2=x^3+1$ for $p \equiv 3 \pmod 4$ or $p\equiv 2
\pmod 3$, respectively.
Let $k=\Fp(T)$  and $\zeta_m\in \F_{p^2-1}^\times$ an element
of order $m$.  
Note $m|p^2-1$ and $m\nmid p-1$, thus $\zeta_m\not\in
\Fp^\times$ and $\F_{p^2}=\Fp[\zeta_m]$. Put $K=\F_{p^2}(T^{1/m})$ and 
$k_2=\F_{p^2}(T)$. Then
$K/k$ is finite Galois with dihedral Galois group $D_{2m}$ of 
order ${2m}$, 
which is  generated by $\tau$ and $c$, where 
\[ c(\zeta_m)=\zeta_m^{-1}, 
\ \ c(T^{1/m})=T^{1/m}, \ \ \tau(T^{1/m})=\zeta_m
T^{1/m},\ \  \tau\in \Gal(K/k_2). \]
We may identify $\Gal(k_2/k)=\Gal(\F_{p^2}/\Fp)=\{1,c\}$ and $c$ acts
on $\Aut(E_0\otimes \F_{p^2})$ by $c(\eta)=\eta^{-1}$. 

Define a $1$-cocycle $\{\xi_\sigma\}\in Z^1(\Gal(K/k), C_m)$ by
\[ \xi_{\tau^i}=\eta^{i}, \quad \xi_{c \tau^i}=\eta^{-i}
d, \quad \forall\, i=0,\dots, m-1, \]
where $d$ is any element in $C_m$, and let $E$ be the corresponding
elliptic curve over $k=\Fp(T)$. Then 
\[ \End^0(E)\simeq \{a\in \End^0(E_0\otimes K)\, |\, \xi_\sigma
\sigma(a)\xi_\sigma^{-1}=a, \ \forall\, \sigma\in \Gal(K/k)\}. \]
Let $a\in \End^0(E)$ be an element. Put $\sigma=\tau$, then one has 
$a\in \Q(\eta)$. Put $\sigma=c$, then 
$c(a)=a$ implies $a\in \Q$. Thus, $\End^0(E)=\Q$.

\subsection{Part (2)} % : the case where $\F_{p^2}\subset k$}
\label{sec:4.3}
Theorem~\ref{thm:end-ss} (2) will follow from the following two
lemmas.

% The following lemma shows 
% the case of Theorem~\ref{thm:end-ss} (2) where 
% $\F_{p^2}\subset k$.
% The restriction to $p>3$ is necessary; 
% see examples in Section~\ref{sec:03}. However, 
% the assumption $p\equiv 1 \pmod {12}$ is not required.  

\begin{lemma}\label{4.1}
  If $p>3$ and the base field $k$ contains $\F_{p^2}$, then the
  endomorphism algebra of any supersingular elliptic curve $E$ over $k$
  is isomorphic to either $B_{p,\infty}$, $\Q(\zeta_4)$, or
  $\Q(\zeta_6)$. In particular, $\End^0(E)\neq \Q$.
\end{lemma}
\begin{proof}
  We know that any supersingular $j$-invariant is contained $\F_{p^2}$
  and that any elliptic curve $E'$ over an \ac field of \ch $p$ has 
  a model defined over $\Fp(j)$. There is a finite Galois extension
  $K/k$, a supersingular elliptic curve $E_0$ over $\F_{p^2}$ and  
  a $K$-isomorphism $E\otimes K\simeq
  E_0\otimes_{\F_{p^2}} K$, that is $E$ is a $K/k$-form of $E_0\otimes
  k$.  Replacing $E_0$ by a form of itself and increasing 
  $K$ if necessarily we may assume that $\End^0(E_0)\simeq
  B_{p,\infty}$. Since all $\Fpbar$-endomorphisms of $E\otimes \Fpbar$
  is defined over $\F_{p^2}$, the group $\Gal(K/k)$ acts trivially on
  $\End(E_0\otimes K)$. As $p>3$, the automorphism group
  $G=\Aut(E_0)$ is abelian and
  $H^1(\Gal(K/k), G)\simeq \Hom(\Gal(K/k),G)$. Let $\xi\in
  \Hom(\Gal(K/k), G)$ be the $1$-cocycle corresponding to $E$.
  %By the same computation as before,
  Similarly, $\End^0(E)$ is isomorphic to the centralizer of the image of
  $\xi$. It follows that $\End^0(E)\simeq  B_{p,\infty}$ or $\Q(\zeta_m)$
  according as ${\rm Im}(\xi)\subset\{\pm 1\}$ or 
  ${\rm Im}(\xi)=C_m$ with $m=4,6$. 
  \qed
\end{proof}

% The proof of Lemma~\ref{4.1} actually shows that $\End^0$ is
% isomorphic to $B_{p,\infty}$, $\Q(\zeta_4)$, or $\Q(\zeta_6)$. 

\begin{lemma}\label{odd}
  Assume that $p\equiv 1\pmod {12}$ and $k\not \supset \F_{p^2}$. Then
  for any supersingular elliptic curve $E$ over $k$, one has
  $\End^0(E)\simeq \Q(\sqrt{-p})$. 
\end{lemma}

\begin{proof}
  Replacing $k$ by a subfield of itself we may assume that $k$ is
  finitely generated over $\Fp$. The algebraic closure $\Fq$ of $\Fp$
  in $k$ has cardinality $q=p^a$ of an odd power of $p$. 
  
  Since $j(E)\in \Fq \cap \F_{p^2}=\Fp$, there is a supersingular
  elliptic curve $E_0$ over $\Fp$, a finite Galois extension $K/k$,
  and a $K$-isomorphism $E\otimes K \simeq E_0\otimes K$ (see the proof
  of Lemma~\ref{4.1}). Particularly 
  $E$ is a $K/k$-form of $E_0\otimes k$. The Frobenius endomorphism
  $\pi_{E_0}$ of $E_0$ satisfies $\pi_{E_0}^2=-p$ as $p>3$. 
  Since the
  Frobenius endomorphism of $E_0\otimes \Fq$ is not in $\Q$, one has
  $\End^0(E_0\otimes \Fq)=\Q(\pi_{E_0}^a)=\Q(\sqrt{-p})$. 
%  As $\Fq$ is algebraically closed in $k$, one obtains 
  By Lemma~\ref{23}, $\End^0(E_0\otimes
  k)=\End^0(E_0\otimes \Fq)=\Q(\sqrt{-p})$.   
  Our assumption of $p$ implies that $\Aut(E_0\otimes K)=\{\pm 1\}$ 
  (see \cite[Table 1.3, p.~117 ]{gross:ht_L87}), which is
  contained in the center of $\End(E_0\otimes K)$. Finally 
  by Lemma~\ref{end} one has
  \begin{equation}
    \label{eq:end}
    \begin{split}
    \End^0(E) & \simeq 
     \{a\in \End^0(E_0\otimes K)\mid \sigma(a)=a, \ \forall\,
        \sigma\in \Gal(K/k) \}\\ 
    & =\End^0(E_0\otimes k)\simeq \Q(\sqrt{-p}). \quad \text{\qed}       
    \end{split}
  \end{equation}

\end{proof}

For the convenience of the reader, 
we make the following table of isogeny classes and endomorphism
algebras of supersingular elliptic curves over finite fields 
(cf. \cite[Chapter 4]{waterhouse:thesis}). 
Here $E$ denotes a  supersingular elliptic curve over $\Fq$, 
$q=p^a$, $\zeta_n:=\exp (2\pi i/n)\in \C$ and 
$\pi$ is the Frobenius endomorphism of $E$, which is represented by 
a Weil $q$-number. 

% Below is the list of the isogeny classes and endomorphism
% algebras of supersingular elliptic curves over $\F_q$. Here $q=p^a$
% is a power of the prime $p$, $\pi$ denotes the relative Frobenius
% endomorphism. 

% We hope the reader would find it intersting that arithmetic of 
% quaternion
% algebras and solutions to the IGP are useful to study abelian
% varieties which are not yet trained. 
 
% Last section we attempt to generalize Statement (A) to more general 
% abelian varieties; see Theorem. We show that any abelian variety is
% isogenous to, up to a finite field extension, a minimal abelian
% variety. However, we could not remove the condition up
% to a finite field extension at this moment though we suspect this 
% is not necessary. 

\begin{center}
\begin{tabular}{ |c|c|c|c| }
\hline
 & \multicolumn{3}{ |c| }{$a$ is even}  \\ \hline
$\pi$ & $\pm p^{a/2}$ & $p^{a/2}\zeta_4$, $p\not \equiv 1 \pmod 4$  & 
  $\pm p^{a/2}\zeta_6, p\not \equiv 1 \pmod 3$ \\ \hline
$\End^0(E)$ & $B_{p,\infty}$ & $\Q(\sqrt{-1})$ & $\Q(\sqrt{-3})$ 
    \\ \hline \hline
  &   \multicolumn{3}{ |c| }{$a$ is odd} \\ \hline
$\pi$ &  $\sqrt{q} \zeta_4$ & $\pm \sqrt{2^{a}}\zeta_8$ 
  & $\pm \sqrt{3^{a}}\zeta_{12}$ \\ \hline
$\End^0(E)$ & $\Q(\sqrt{-p}) $ & $\Q(\sqrt{-1})$ & 
  $\Q(\sqrt{-3})$ \\ \hline \hline
\end{tabular}
\end{center}

\ \\

\begin{center}
  {\bf Erratum to \cite{yu:endo}}
\end{center} \

Very unfortunately the main result Theorem 1.2 in \cite{yu:endo} 
is not correct as stated. We apologize to the Journal for overlooks. 
Theorem 1.1 is an equivalent statement of 
Theorem 1.2. 
Our purpose of formulating this theorem was to introduce invariants 
which could control the finiteness. However, this approach turns out
not working well. The errors occur at the proofs of two reduction steps
Lemmas 2.2 and 2.4, loc.cit.

To see that the second assertion (2.2) of Lemma 2.2 is false, 
take a supersingular elliptic curve
$E$ over $\F_7$ whose endomorphism ring is $\Z[\sqrt{-7}]$, which has
index $2$ in the maximal order. However, the endomorphism ring of
$E\otimes \F_{p^2}$ is a maximal order of the quaternion algebra over
$\Q$ ramified at $\{7,\infty\}$. 
      
To correct (2.2) of Lemma 2.2 one should impose an additional
condition that $\End(A)$ is contained in the center of $\End(A_{k'})$, 
where $A_{k'}=A\otimes_{k} k'$. \\
Proof: Since $\End(A\otimes k^{\rm perf})=\End(A)$, where $k^{\rm
  perf}$ is the perfect closure of $k$, we may assume that both $k$ and
$k'$ are perfect by replacing them by their perfect closure. 
Choose a maximal order $O_1$ in the endomorphism algebra
   $\End^0(A)$ containing $\End(A)$. Let $M\subset M'$ be the \dieu
   modules of $A$ and $A_{k'}$, respectively. Put $M_1:=O_1 M$
   and let $A\to A_1$ be the isogeny (unique up to isomorphism)
   realizing $M\to M_1$. Write $\alpha:\End^0(A)\isoto \End^0(A_1)$
   for the isomorphism. Then $\alpha(O_1)=\End(A_1)$ and it is
   contained in $\End(A_{1k'})$. Now from our assumption that 
   $O_1$ commutes with
   $\End(A_{k'})$, the latter leaves the \dieu module 
  $M_1':=M(A_1\otimes k')=O_1
   M'$ stable. This shows that $\alpha(\End(A_{k'}))\subset
   \End(A_{1k'})$. An element $x\in O_{1p}=O_1\otimes \Zp$ 
   lies in $O_p$ if and only if $xM\subset M$, or the same 
   $xM'\subset M'$. 
   Thus, $\alpha(O_{1p})\cap \alpha(\End(A_{k'})\otimes
   \Zp)=\alpha(\End(A)\otimes \Zp)$ and we have an injective map
\[ \alpha: \frac{O_{1p}}{\End(A)\otimes \Zp} \to \frac
{\End(A_{1k'})\otimes \Zp}{\alpha(\End(A_{k'})\otimes \Zp)}.\]
Using the Tate modules instead of \dieu modules, we get an injective map
\[ \alpha: \frac{O_{1}}{\End(A)} \to \frac
{\End(A_{1k'})}{\alpha(\End(A_{k'}))}. \quad \text{\qed}\]

The statement of Lemma 2.4 should be corrected by imposing an additional
condition that $\End(A)\otimes \Zp$ is contained in the center of
$\End(A[p^\infty])$. This condition ensures that $\End(A[p^\infty])$
is contained in $\End(A'[p^\infty])$ in the proof of Lemma 2.4. \\

\npr {\bf Counterexample to Theorem 1.2.} 
Choose a simple ordinary abelian surface $A_0$ over
$\Fpbar$. Modifying $A_0$ by an isogeny, we can assume that 
its endomorphism ring $O:=\End(A_0)$ is the 
maximal order in a CM field $K:=\End^0(A_0)$ \cite[Theorem
1.3]{yu:mo}. Let $M_0$ be the \dieu
module of $A_0$, and let $M_0=M_0^0\oplus M_0^1$ be the isotypic
decomposition with slopes $0$ and $1$, respectively. The endomorphism
ring $\End(M_0)=\End(M_0^0)\times \End(M_0^1)$ also decomposes. 
We fix a $\Zp$-basis for the skeleton of
each $M_0^i$ (\cite[(3.4), p.~363]{yu:endo}) and have
$\End(M_0^i)=\Mat_2(\Zp)$.    
The action of the order $O_p:=O\otimes \Zp$ leaves each component
$M_0^i$ invariant. As $O$ is the maximal order in $K$, 
one has $O_p=O_p^0\times
O_p^1$ and $O_p^i\subset \End(M_0^i)$ for $i=1,2$. Consider a \dieu
submodule $M_1=M_1^0\oplus M_1^1\subset M_0$ with $M_1^1=M_0^1$, and
let $A_1\to A_0$ be the isogeny realizing the inclusion $M_1\subset
M_0$. Since $M_1^0$ is isomorphic to $M_0^0$, it is of the form
$gM^0_0$, where $g\in \Aut(M_0^0[1/p])=\GL_2(\Qp)$. Note
$\End(A_1)$ is contained in $O$ as $O$ is the unique maximal order in
$K$. Therefore, $\End(A_1)\otimes \Zp=O^0_{1p}\times O_p^1$ with
$O^0_{1p}=O_p^0 \cap g \End(M_0^0) g^{-1}=O_p^0\cap g \Mat_2(\Zp) g^{-1}$. Thus,
\[ [O_p:\End(A_1)\otimes \Zp]=[O_p^0: O_p^0\cap g \Mat_2(\Zp) g^{-1}]. \]
This index is unbounded for all $g\in \GL_2(\Qp)$.

\begin{lemma*}
  Let $A$ be an abelian variety over a perfect field $k$, and let
  $A\to A^{\rm min}$ be its minimal isogeny. Let $M\subset M_0$ be
  the \dieu modules of this isogeny and put $\ol M:=M\otimes_W W(\bar
  k)$ and $\ol M_0:=M_0\otimes_W W(\bar k)$. Let
\[ O:=\End(A), , \Lambda:=\End(M),
, \ol \Lambda:=\End(\ol M), \]
\[ O_0:=\End(A^{\rm min}), \quad \Lambda_0:=\End(M_0),\quad
\ol \Lambda_0:=\End(\ol M_0). \]
Then we have inclusions
\[  \frac{O_0}{O}= \frac{O_0\otimes \Zp}{O\otimes \Zp}\subset
\frac{\Lambda_0}{ \Lambda} \subset  
\frac{\ol \Lambda_0}{\ol \Lambda}. \] 
\end{lemma*}
\begin{proof}
  By Lemma~\ref{M.4} of this note, 
  the minimal \dieu module $M_0$ is generated by
  $M$ and certain operators only involving powers of $F$, $V$, $p$. As
  elements of $\End(M)$ commute with these operators, one has
  $\End(M)\subset \End(M_0)$ and $O\subset O_0$, 
  and hence $\End(\ol M)\subset \End(\ol
  M_0)$. Now the lemma follows from $O_0\cap \End(M)=O$ and
  $\Lambda_0\cap \End(\ol M)=\End(M)$. \qed  
\end{proof}
  
For any field $k$ and integer $g\ge 1$, denote by $\calA_g^{un}(k)$
the set of isomorphism classes of $g$-dimensional abelian varieties
over $k$. 

\begin{cor*}
  Let $k$ be any field of \ch $p$ and $g\ge 1$ a fixed integer. 
  Let $\calB\subset
  \calA^{un}_g(k)$ be a subset which is stable under isogeny. Then the
  following two statements are equivalent.

(a) There are only finitely many isomorphism classes of $\Zp$-orders 
    $\End(A)\otimes \Zp$ for all abelian varieties $A\in \calB$. 

(b) There are only finitely many isomorphism classes of $\Zp$-orders 
    $\End(A)\otimes \Zp$ for all minimal abelian varieties $A\in \calB$.
\end{cor*}
\begin{proof}
  Again we may assume that $k$ is perfect as $\End(A_{k^{\rm
      perf}})=\End(A)$, and $A$ is minimal if and only if so is $A_{k^{\rm
      perf}}$. 
  By \cite[Theorem 2.5]{yu:endo}, $|\ol \Lambda_0/\ol \Lambda|\le
  p^N$ for an integer $N$ which depends only on $g$. Then the
  corollary follows from {\bf Lemma}. \qed 
\end{proof}

We modify Theorem 1.2 as the following weaker result, which suffices
to imply Corollary 1.3 and Theorem 1.4 (Deuring) in \cite{yu:endo}.
See \cite[Definition 2.1]{chai-oort:hyper} for the definition of 
hypersymmetric abelian varieties. 

\begin{thm*}
  Notations as in {\bf Corollary}, let $\calI\subset \calA_g^{un}(k)$
  be one isogeny class. There are only finitely
  many isomorphism classes of $\Zp$-orders  
    $\End(A)\otimes \Zp$ for all abelian varieties $A\in \calI$
    provided one of the
    following conditions holds:
  
(a) There are only finitely many isomorphism classes of $\Zp$-orders 
    $\End(A)\otimes \Zp$ for all minimal abelian varieties $A\in \calI$.

(b) The field $k$ is \ac and any member $A$ in $\calI$ is
hypersymmetric.    
\end{thm*}

Condition (b) implies condition (a). 

% Now we show that the original statement of Theorem 1.2 is incorrect. 
% Choose a simple ordinary abelian surface $A_0$ over
% $\Fpbar$. Modifying $A_0$ by an isogeny, we can assume that 
% $O:=\End(A_0)$ is a maximal order a CM field $K=\End^0(A)$. 

% However, this error does not affect the main results Theorems 1.1 and
% 1.2. The purpose of Lemma 2.2 is to reduce Theorem 1.1 to the case
% where the ground field $k$ is algebraically closed. We change the
% argument of this step as follows.

% Let $V$ be a finite-dimensional $\Q$-vector space with a
% non-degenerate symmetric bilinear form $\varphi:V\otimes V\to \Q$. For
% any full lattice $L$, define $d(L):=[L^\vee:L]\in \Q^\times_+ $, where
% $L^\vee$ is the 
% dual lattice of $L$. The relative index is defined by
% $[L_1:L_2]:=[L_1:L_3] [L_2:L_3]^{-1}$ for any full sublattice $L_3$ in
% $L_1\cap L_2$. Clearly, $v_p(d(L))=v_p (d(L\otimes \Z_p))$. 

% For any finite-dimensional semi-simple $\Q$-algebra $B$, we consider 
% the non-degenerate pairing $\varphi$ defined by
% $\varphi(a,b):=\tr(\ell(a)\ell(b);B)$, where $\ell(a)$ is the
% left multiplication by $a$ on $B$. Thus, for any order $O$ in $B$, one
% has defined $d(O)\in \bbN$. If $R$ is a maximal order containing $O$,
% then $d(O)=d(R)[R:O]^2={\rm ci}(O)^2 d(R)$. 

% \begin{lemma}
%   Let $B'\subset B$ be a semi-simple subalgebra and $O'=B'\cap
%   O$. Then $d(O')|d(O)$. 
% \end{lemma}

\section*{Acknowledgments}

The author is grateful to Ching-Li Chai for helpful discussions on
Theorem~\ref{1.1}. 
%on Propositon~\ref{21}.
The present work is done while the author's stay in the Max-Planck-Institut
f\"ur Mathematik. He is grateful to the Institut for kind hospitality
and excellent working environment. 
The author is partially supported by the grants 
MoST 100-2628-M-001-006-MY4 and 103-2918-I-001-009.

% \section{Minimal isogenies over perfect fields}
% \label{sec:06}

% \section{Minimal isogenies for abelian varieties over finite fields.}  
% \label{sec:07}

% Describe the lattice structure...

% Call an abelian variety $A$ distinguished over finite field if $\End(A)$
% contains the maximal order of its center. 
% Hull(A) of an abelian variety over a finite field.

% Jacobinski type result: If $Hull(A)\sim Hull(B)$ and $A$ and $B$ are in
% the same genus, then $A\simeq B$.   

\def\jams{{\it J. Amer. Math. Soc.}} %ok
\def\invent{{\it Invent. Math.}} %ok
\def\ann{{\it Ann. Math.}} %ok
\def\ihes{{\it Inst. Hautes \'Etudes Sci. Publ. Math.}} %ok

\def\ecole{{\it Ann. Sci. \'Ecole Norm. Sup.}}
\def\ecole4{{\it Ann. Sci. \'Ecole Norm. Sup. (4)}} 
\def\mathann{{\it Math. Ann.}} %ok
\def\duke{{\it Duke Math. J.}} %ok
\def\jag{{\it J. Algebraic Geom.}} %ok
\def\advmath{{\it Adv. Math.}}
\def\compos{{\it Compositio Math.}} %ok
\def\ajm{{\it Amer. J. Math.}} 
\def\crelle{{\it J. Reine Angew. Math.}}
\def\plms{{\it Proc. London Math. Soc.}}
\def\jussieu{{\it J. Inst. Math. Jussieu}} 
\def\grenoble{{\it Ann. Inst. Fourier (Grenoble)}}
\def\imrn{{\it Int. Math. Res. Not.}}
\def\tams{{\it Trans. Amer. Math. Sci.}}
\def\mrl{{\it Math. Res. Lett.}}
\def\cras{{\it C. R. Acad. Sci. Paris S\'er. I Math.}} %ok
\def\mathz{{\it Math. Z.}} %ok
\def\cmh{{\it Comment. Math. Helv.}}
\def\docmath{{\it Doc. Math. }}
\def\asian{{\it Asian J. Math.}}
\def\acta{{\it Acta Math.}}
\def\indiana{{\it Indiana Univ. Math. J.}}

\def\acad{{\it Proc. Nat. Acad. Sci. USA}}

\def\jlms{{\it J. London Math. Soc.}}
\def\blms{{\it Bull. London Math. Soc.}}
\def\manmath{{\it Manuscripta Math.}} %ok
\def\jnt{{\it J. Number Theory}} %ok 
\def\ijm{{\it Israel J. Math.}}
\def\ja{{\it J. Algebra}} %ok
\def\pams{{\it Proc. Amer. Math. Sci.}}
\def\smfmemoir{{\it Bull. Soc. Math. France, Memoire}}
\def\bsmf{{\it Bull. Soc. Math. France}}
\def\sb{{\it S\'em. Bourbaki Exp.}}
\def\jpaa{{\it J. Pure Appl. Algebra}}
\def\jems{{\it J. Eur. Math. Soc. (JEMS)}}
\def\jtokyo{{\it J. Fac. Sci. Univ. Tokyo}}
\def\cjm{{\it Canad. J. Math.}}
\def\jaums{{\it J. Australian Math. Soc.}}
\def\pspm{{\it Proc. Symp. Pure. Math.}}
\def\ast{{\it Ast\'eriques}}
\def\pamq{{\it Pure Appl. Math. Q.}}
\def\nagoya{{\it Nagoya Math. J.}}
\def\forum{{\it Forum Math. }}
\def\tjm{{\it Taiwanese J. Math.}}
\def\rt{{\it Represent. Theory}}
\def\bordeaux{{\it J. Th\'eor. Nombres Bordeaux}}
\def\ijnt{{\it Int. J. Number Theory}}
\def\jmsj{{\it J. Math. Soc. Japan}}
\def\rims{{\it Publ. Res. Inst. Math. Sci.}}
\def\ca{{\it Comm. Algebra}}
\def\osaka{{\it Osaka J. Math.}}
\def\bams{{\it Bull. Amer. Math. Soc.}}

\def\tp{{To appear in }}

\newcommand{\princeton}[1]{Ann. Math. Studies #1, Princeton
  Univ. Press}

\newcommand{\LNM}[1]{Lecture Notes in Math., vol. #1, Springer-Verlag}

\end{document}